\documentclass[12pt,a4paper]{amsart}

\usepackage[utf8]{inputenc}
\usepackage{enumerate,bbm,titletoc,a4,mathtools,mathrsfs,amssymb,xspace,amsthm}
\usepackage[usenames,dvipsnames]{xcolor}
\usepackage[pagebackref,colorlinks,linkcolor=BrickRed,citecolor=OliveGreen,urlcolor=black,hypertexnames=true]{hyperref}

\DeclareMathOperator{\GL}{GL}

\DeclareMathOperator{\re}{re}
\DeclareMathOperator{\im}{im}

\DeclareMathOperator{\opm}{M}
\DeclareMathOperator{\ops}{S}

\newcommand{\ve}{\varepsilon}

\newcommand{\de}{\delta}

\newcommand{\N}{\mathbb{N}}

\newcommand{\R}{\mathbb{R}}
\newcommand{\C}{\mathbb{C}}

\newcommand{\cD}{\mathcal{D}}

\newcommand{\fp}{\mathfrak{p}}

\newcommand{\kk}{\mathbbm k}
\newcommand{\bkk}{\overline{\kk}}

\newcommand{\gX}{\Omega}

\newcommand*{\mat}[1]{\opm_{#1}(\kk)}
\newcommand*{\mc}[1]{\opm_{#1}(\C)}
\newcommand*{\sm}[1]{\ops_{#1}(\R)}
\newcommand{\all}{\mathbb{S}^d}

\newcommand{\Langle}{\mathop{<}\!}
\newcommand{\Rangle}{\!\mathop{>}}
\newcommand{\ulx}{\underline{x}}
\newcommand{\uly}{\underline{y}}
\newcommand{\ult}{\underline{t}}

\newcommand{\px}{\R\!\Langle \ulx\Rangle}
\newcommand{\ax}{\kk\!\Langle \ulx\Rangle}
\newcommand{\bx}{\bkk\!\Langle \ulx\Rangle}
\newcommand{\py}{\C\!\Langle \uly,\uly^*\Rangle}

\makeatletter
\def\moverlay{\mathpalette\mov@rlay}
\def\mov@rlay#1#2{\leavevmode\vtop{
		\baselineskip\z@skip \lineskiplimit-\maxdimen
		\ialign{\hfil$#1##$\hfil\cr#2\crcr}}}
\makeatother

\newcommand{\plangle}{\moverlay{(\cr<}}
\newcommand{\prangle}{\moverlay{)\cr>}}
\newcommand{\rx}{\kk\plangle \ulx \prangle}

\newcommand{\eig}{eigenlevel set\xspace}
\newcommand{\eigs}{eigenlevel sets\xspace}
\newcommand{\Eigs}{Eigenlevel sets\xspace}

\newcommand{\wqc}{locally quasiconvex\xspace}
\newcommand{\Wqc}{Locally quasiconvex\xspace}

\newtheorem{thm}{Theorem}[section]
\newtheorem{lem}[thm]{Lemma}
\newtheorem{cor}[thm]{Corollary}
\newtheorem{prop}[thm]{Proposition}
\newtheorem{thmA}{Theorem}

\theoremstyle{definition}

\newtheorem{exa}[thm]{Example}
\theoremstyle{remark}
\newtheorem{rem}[thm]{Remark}

\numberwithin{equation}{section}

\title[Free Bertini's theorem and applications]{Free Bertini's theorem and applications}

\author[J. Vol\v{c}i\v{c}]{Jurij Vol\v{c}i\v{c}}
\address{Jurij Vol\v{c}i\v{c}, Department of Mathematics, Texas A\&M University}
\email{volcic@math.tamu.edu}

\date{\today}

\begin{document}
	
\begin{abstract}
The simplest version of Bertini's irreducibility theorem states that the generic fiber of a non-composite polynomial function is an irreducible hypersurface. The main result of this paper is its analog for a free algebra: if $f$ is a noncommutative polynomial such that $f-\lambda$ factors for infinitely many scalars $\lambda$, then there exist a noncommutative polynomial $h$ and a nonconstant univariate polynomial $p$ such that $f=p\circ h$. Two applications of free Bertini's theorem for matrix evaluations of noncommutative polynomials are given. An {\it \eig} of $f$ is the set of all matrix tuples $X$ where $f(X)$ attains some given eigenvalue. It is shown that \eigs of $f$ and $g$ coincide if and only if $fa=ag$ for some nonzero noncommutative polynomial $a$. The second application pertains quasiconvexity and describes polynomials $f$ such that the connected component of $\{X \text{ tuple of symmetric $n\times n$ matrices}\colon \lambda I\succ f(X) \}$ about the origin is convex for all natural $n$ and $\lambda>0$. It is shown that such a polynomial is either everywhere negative semidefinite or the composition of a univariate and a convex quadratic polynomial.
\end{abstract}

\maketitle


\section{Introduction}

Bertini's irreducibility theorem (see e.g. \cite[Theorem 2.26]{Sha}) is a fundamental result with a rich history \cite{Kle} and omnipresent in algebraic geometry. When applied to a multivariate polynomial function $f$ over a an algebraically closed field, it states that the hypersurface $\{f=\lambda\}$ is irreducible for all but finitely many values $\lambda$ unless $f$ is a composite with a univariate polynomial. This particular case is significant in its own right in commutative algebra, and has been extensively studied and generalized \cite{Sch,BDN}. In this paper we prove its analog for a free associative algebra and derive consequences of interest for free analysis \cite{KVV} and free real algebraic geometry \cite{HKM0,BPT}.

Let $\kk$ be a field and $d\in\N$. Let $\ax$ be the free associative $\kk$-algebra in freely noncommuting variables $\ulx=(x_1,\dots,x_d)$. Its elements are called {\bf noncommutative polynomials}. We say that $f$ {\bf factors} in $\ax$ if $f=f_1f_2$ for some nonconstant $f_1,f_2\in\ax$. Otherwise, $f$ is {\bf irreducible} over $\kk$. A nonconstant $f\in\ax$ is {\bf composite} (over $\kk$) if there exist $h\in\ax$ and a univariate polynomial $p\in\kk[t]$ such that $\deg p>1$ and $f=p\circ h=p(h)$. Our first main result is the free algebra analog of a special case of the classical Bertini's (irreducibility) theorem.

\begin{thmA}[Free Bertini's theorem]\label{t:a}
If $f\in\ax\!\setminus\kk$ is not composite, then $f-\lambda$ is irreducible over $\bkk$ for all but finitely many $\lambda\in\bkk$.
\end{thmA}

See Theorem \ref{t:bertini} for a more comprehensive statement and proof. Next we apply Theorem \ref{t:a} to matrix evaluations of noncommutative polynomials. Let $f\in\ax$. Given $X\in\mat{n}^d$ let $f(X)\in\mat{n}$ be the evaluation of $f$ at $X$. The {\bf \eig} of $f$ at $\lambda\in\kk$ is
$$L_\lambda(f)=\bigcup_{n\in\N} \left\{X\in\mat{n}^d\colon \lambda\ \text{is an eigenvalue of}\ f(X)\right\}.$$
In terms of \cite{KV,HKV,HKV1}, eigenlevel sets are free loci of polynomials $\lambda-f$, which have been intensively studied for their implications to domains of noncommutative rational functions \cite{KVV1}, factorization in a free algebra \cite{HKV,HKV1} and matrix convexity \cite{BPT,HKM,DDOSS}. Using Theorem \ref{t:a} we derive the following algebraic certificate for inclusion of \eigs (see Theorem \ref{t:eig} for the proof).

\begin{thmA}\label{t:b}
Let $\kk$ be an algebraically closed field of characteristic $0$ and $f,g\in\ax$. Then \eigs of $f$ are contained in \eigs of $g$ if and only if there exist nonzero $a,h\in\ax$ and $p\in\kk[t]$ such that $g=p(h)$ and $fa=ah$.
\end{thmA}

Lastly we turn to noncommutative polynomials describing convex matricial sets. Let $\sm{n}\subseteq \opm_{n}(\R)$ denote the subspace of symmetric matrices. In \cite{BM}, a symmetric $f\in\px$ with $f(0)=0$ is called {\it quasiconvex} if for every $n\in\N$ and positive definite $A\in\sm{n}$, the set
\begin{equation}\label{e:qc}
\{X\in\sm{n}^d\colon A-f(X)\ \text{is positive semidefinite}\}
\end{equation}
is convex; see \cite[Subsection 1.1]{BM} for the relation with the classical (commutative) notion of quasiconvexity. Furthermore, in \cite[Theorem 1.1]{BM} the authors showed that every quasiconvex polynomial is either convex quadratic or minus a sum of hermitian squares (i.e., $-f=\sum_m h^*_mh_m$ for some $h_m\in\px$, in which case the set \eqref{e:qc} equals $\sm{n}^d$ for every $A\succ0$ and $n\in\N$).

To relate quasiconvexity more closely to the notion of a free semialgebraic set \cite{HM,HKM0} in free real algebraic geometry, 
we say that a symmetric $f\in\px$ with $f(0)=0$ is {\bf \wqc} if there exists $\ve>0$ such that the connected component of
$$\{X\in\sm{n}^d\colon \lambda I-f(X)\ \text{is positive definite}\}$$
containing the origin is convex for every $n\in\N$ and $\lambda\in (0,\ve)$.

\begin{thmA}\label{t:c}
If $f\in\px$ is \wqc, then either $-f$ is a sum of hermitian squares or $f=p(\ell_0+\ell_1^2+\cdots+\ell_m^2)$ for some $p\in\R[t]$ and linear $\ell_0,\dots,\ell_m\in\px$.
\end{thmA}

A precise biconditional statement is given in Theorem \ref{t:wqc} below.

\subsection*{Acknowledgments}
The author thanks George Bergman for enlightening correspondence and Igor Klep for valuable suggestions.

\section{Preliminaries}\label{s2}

We start with reviewing certain notions and technical results from the factorization theory of P. M. Cohn \cite{Coh} that will be used throughout the paper. Most of this theory is based on the fact that $\ax$ is a free ideal ring (see e.g. \cite[Corollary 2.5.2]{Coh}), which will be implicitly used when referring to the existing literature.

Noncommutative polynomials $f_1,f_2\in\ax$ are {\bf stably associated} \cite[Section 0.5]{Coh} if there exist $P_1,P_2\in\GL_2(\ax)$ such that $f_2\oplus 1 =P_1(f_1\oplus 1)P_2$. Equivalently \cite[Proposition 0.5.6 and Theorem 2.3.7]{Coh}, there exist $g_1,g_2\in\ax$ such that $f_1,g_2$ are left coprime, $g_1,f_2$ are right coprime, and
\begin{equation}\label{e:sa0}
f_1g_1=g_2f_2.
\end{equation}
Here left (right) coprime refers to the absence of a non-invertible common left (right) factor. The importance of stable association steams from the fact that factorization of a noncommutative polynomial into irreducible factors is unique up to stable association of factors \cite[Proposition 3.2.9]{Coh}. The following finiteness result was first proved by G. M. Bergman in his doctoral thesis.

\begin{prop}[Bergman, {\cite[Exercise 2.8.8]{Coh}}]\label{p:berg}
Given $f\in\ax$, there are (up to a scalar multiple) only finitely many polynomials stably associated to $f$.
\end{prop}

We will also require degree bounds on ``witnesses'' of stable association in \eqref{e:sa0}.

\begin{lem}\label{l:cohn}
If $f_1,f_2\in\ax\setminus\kk$ are stably associated, then $\deg f_1=\deg f_2$ and there exist nonzero $g_1,g_2\in\ax$ such that $\deg g_i<\deg f_i$ and $f_1g_1=g_2f_2$.
\end{lem}

\begin{proof}
Following \cite[Section 2.7]{Coh}, continuant polynomials $\fp_n\in \kk\!\Langle y_1,\dots,y_n\Rangle$ are recursively defined as
$$\fp_0=1,\qquad \fp_1=y_1,\qquad \fp_n=\fp_{n-1}y_n+\fp_{n-2} \ \ \text{for}\ \ n\ge2.$$
By  \cite[Proposition 2.7.6]{Coh} there exist $\alpha_1,\alpha_2\in\kk$ and $a_1,\dots,a_r\in\ax$ such that
$$f_1=\alpha_1 \fp_r(a_1,\dots,a_r),\qquad f_2=\alpha_2 \fp_r(a_r,\dots,a_1)$$
and $a_i$ are nonconstant for $1<i<r$. If $a_r=0$, then
$$\fp_r(a_1,\dots,a_r)=\fp_{r-2}(a_1,\dots,a_{r-2}),\quad \fp_r(a_r,\dots,a_1)=\fp_{r-2}(a_{r-2},\dots,a_1).$$
If $a_r\in\kk\setminus\{0\}$, then an easy manipulation of the recursive relation for $\fp_n$ yields
$$\fp_r(a_1,\dots,a_r)=a_r \fp_{r-1}(a_1,\dots,a_{r-1}+\tfrac{1}{a_r}),\quad
\fp_r(a_r,\dots,a_1)=a_r \fp_{r-1}(a_{r-1}+\tfrac{1}{a_r},\dots,a_1).$$
Analogous conclusions hold for $a_1\in\kk$. Hence there exist $\beta_1,\beta_2\in\kk$ and nonconstant $b_1,\dots,b_s\in\ax$ such that 
$$f_1=\beta_1 \fp_s(b_1,\dots,b_s),\qquad f_2=\beta_2 \fp_s(b_s,\dots,b_1).$$
Since
$$\fp_s(b_1,\dots,b_s)\fp_{s-1}(b_{s-1},\dots,b_1) = \fp_{s-1}(b_1,\dots,b_{s-1})\fp_s(b_s,\dots,b_1)$$
holds by \cite[Lemma 2.7.2]{Coh} and the degree of a continuant polynomial in nonconstant arguments equals the sum of degrees of its arguments by the recursive relation,
$$b_1=\tfrac{1}{\beta_1}\fp_{s-1}(b_{s-1},\dots,b_1),\quad b_2=\tfrac{1}{\beta_2}\fp_{s-1}(b_1,\dots,b_{s-1})$$
satisfy $\deg b_1=\deg b_2<\deg f_1=\deg f_2$.
\end{proof}

\begin{rem}
While probably known to the specialists for factorization in free algebras, Lemma \ref{l:cohn} implies that checking whether two irreducible polynomials are stably associated corresponds to solving a (finite) linear system. 
\end{rem}

 Let $\gX^{n}=(\gX_1^{n},\dots,\gX_d^{n})$ be a tuple of {\bf generic $n\times n$ matrices} whose $dn^2$ entries are commuting independent variables are viewed as coordinates of the affine space $\mat{n}^d$.

\begin{lem}[{\cite[Lemma 2.2]{HKV1}}]\label{l:det}
If $f\in\ax$ is nonconstant, then $\det f(\gX^{(n)})$ is nonconstant for large enough $n\in\N$.
\end{lem}

\section{Free Bertini's theorem}\label{s3}

In this section we prove our first main result (Theorem \ref{t:bertini}). First we show that a certain linear equation in a free algebra has a unique solution (up to a scalar multiple).

\begin{lem}\label{l:1}
Let $f,g\in\ax$ and assume $f$ is not composite. If nonzero $\alpha\in\kk$ and $b_1,b_2\in\ax$ satisfy
$$fb_1=b_1g,\quad fb_2=\alpha b_2 g,\quad \deg b_1=\deg b_2<\deg f,$$
then $\alpha=1$ and $b_2\in \kk b_1$.
\end{lem}

\begin{proof}
Since $f$ is not composite, its centralizer in $\ax$ equals $\kk[f]$ by \cite[Theorem 5.3]{Ber}. Therefore its centralizer in $\rx$, the universal skew field of fractions of $\ax$ (see \cite[Chapter 7]{Coh} for more information), equals $\kk(f)$ by \cite[Theorem 7.9.8 and Proposition 3.2.9]{Coh}. Since
$$b_1^{-1}fb_1=g=\alpha^{-1}b_2^{-1}fb_2,$$
we have $\det f(\gX^{(n)})=\alpha^{-n}\det f(\gX^{(n)})$ for large enough $n$ by Lemma \ref{l:det}, so $\alpha=1$ and $b_2b_1^{-1}\in\rx$ commutes with $f$. Hence there exist univariate coprime polynomials $q_1,q_2\in\kk[t]$ such that $b_2b_1^{-1}=q_1(f)^{-1}q_2(f)$, and consequently $q_1(f)b_2=q_2(f)b_1$. Let $b\in\ax$ be such that $b_i=c_ib$ for right coprime $c_1,c_2\in\ax$. Then
$$q_1(f)c_2=q_2(f)c_1,$$
so $q_1(f)$ and $c_1$ are stably associated. Therefore $\deg q_1(f)=\deg c_1$ by Lemma \ref{l:cohn}. Since the degree of $q_1(f)$ is either $0$ or at least $\deg f$, and $\deg c_1\le\deg b_1<\deg f$, we conclude $\deg c_1=0$. Hence $c_1,c_2$ are (nonzero) scalars.
\end{proof}

The proof of free Bertini's theorem is based on Bergman's centralizer theorem \cite{Ber}. While otherwise inherently different from ours, Stein's proof of (the special case of) classical Bertini's theorem in two commuting variables \cite{Ste} also uses ``centralizers'' with respect to the Poisson bracket on $\kk[t_1,t_2]$.

\begin{thm}\label{t:bertini}
Let $\bkk$ be the algebraic closure of a field $\kk$. The following are equivalent for $f\in\ax\!\setminus\kk$:
\begin{enumerate}[(i)]
	\item $f-\lambda$ factors in $\bx$ for infinitely many $\lambda\in\bkk$;
	\item $f-\lambda$ factors in $\bx$ for all $\lambda\in\bkk$;
	\item the centralizer of $f$ in $\ax$ is strictly larger than $\kk[f]$;
	\item $f$ is composite over $\kk$.
\end{enumerate}
\end{thm}

\begin{proof}
Implications (iv)$\Rightarrow$(ii)$\Rightarrow$(i) are clear, and (iii)$\Rightarrow$(iv) is a restatement of \cite[Theorem 5.3]{Ber}. Thus it suffices to prove (i)$\Rightarrow$(iii).

Let $\Lambda\subseteq\bkk$ be an infinite set of $\lambda$ such that $f-\lambda$ factors in $\bx$. For each such $\lambda$ there exist nonconstant $p_\lambda,q_\lambda\in\bx$ such that $f-\lambda=p_\lambda q_\lambda$. Observe that
\begin{equation}\label{e:lam}
fp_\lambda=(\lambda+p_\lambda q_\lambda)p_\lambda=p_\lambda(q_\lambda p_\lambda+\lambda)
\end{equation}
for all $\lambda\in\Lambda$. Since $\deg p_\lambda<\deg f$ for all $\lambda\in\Lambda$, there exists an infinite subset $\Lambda_0\subseteq\Lambda\setminus\{0\}$ and $\de<\deg f$ such that $\deg p_\lambda=\de$ for all $\lambda\in\Lambda_0$. 
Furthermore, $\lambda+p_\lambda q_\lambda,p_\lambda$ are left coprime and $p_\lambda,q_\lambda p_\lambda+\lambda$ are right coprime whenever $\lambda\neq0$. Therefore $f$ and $q_\lambda p_\lambda+\lambda$ are stably associated for every $\lambda\in\Lambda_0$. By Proposition \ref{p:berg}, there are (up to a scalar multiple) only finitely many polynomials stably associated to $f$. Hence there exist distinct $\mu,\nu\in\Lambda_0$ such that $q_\nu p_\nu+\nu$ is a scalar multiple of $q_\mu p_\mu+\mu$. Suppose $f$ is not composite over $\bkk$. Then $p_\nu$ is be a scalar multiple of $p_\mu$ by \eqref{e:lam} and Lemma \ref{l:1}. However, this is impossible since
$$p_\mu q_\mu-p_\nu q_\nu=\nu-\mu\in\bkk\setminus\{0\}.$$
Therefore $f$ is composite over $\bkk$. In particular,
$$U:=\{p\in\bx\colon \deg p<\deg f,\ p(0)=0,\ fp-pf=0 \}\neq\{0\}.$$
But $U$ is a subspace given by equations over $\kk$, so $U\cap\ax\neq\{0\}$. Now (iii) follows because $U\cap\kk[f]=\{0\}$ and $U$ is contained in the centralizer of $f$ in $\ax$.
\end{proof}

A slightly stronger version holds for homogeneous polynomials.

\begin{cor}
Let $f\in\ax\setminus\kk$ be homogeneous. Then $f-1$ factors in $\bx$ if and only if $f=f_0^n$ for some $n>1$ and homogeneous $f_0\in\ax$.
\end{cor}

\begin{proof}
If $f-1$ factors in $\bx$, then $f-\lambda^{\deg f}$ factors in $\bx$ for every $\lambda\in\bkk$ because it is up to a linear change of variables equal to $f(\lambda x)-\lambda^{\deg f}=\lambda^{\deg f}(f-1)$. Therefore $f$ is composite by Theorem \ref{t:bertini}, and furthermore a power by homogeneity.
\end{proof}

\section{\Eigs}\label{s4}

Throughout this section let $\kk$ be an algebraically closed field of characteristic $0$. Recall the definition of the \eig of $f$ at $\lambda$,
$$L_\lambda(f)=\bigcup_{n\in\N} \left\{X\in\mat{n}^d\colon \lambda\ \text{is an eigenvalue of}\ f(X)\right\}.$$
In the terminology of \cite{HKV,HKV1}, $L_\lambda(f)$ is the free locus of $f-\lambda$. Combined with existing irreducibility results for free loci of noncommutative polynomials \cite{HKV,HKV1}, free Bertini's theorem becomes a geometric statement about \eigs.

\begin{cor}\label{c:irr}
If $f\in\ax\!\setminus\kk$ is not composite, then there exists $N\in\N$ such that for all but finitely many $\lambda\in\kk$,
\begin{equation}\label{e:fl}
\left\{X\in\mat{n}^d\colon \lambda\ \text{is an eigenvalue of}\ f(X)\right\}
\end{equation}
is a reduced and irreducible hypersurface in $\mat{n}^g$ for all $n\ge N$. 
\end{cor}

\begin{proof}
By Theorem \ref{t:bertini}, there is a cofinite subset $\Lambda$ of $\kk\setminus\{f(0)\}$ such that $f-\lambda$ is irreducible over $\kk$ for $\lambda\in\Lambda$. By \cite[Theorem 4.3]{HKV} for each $\lambda\in\Lambda$ there exists $N_\lambda\in\N$ such that the hypersurface \eqref{e:fl} is reduced and irreducible for every $n\ge N_\lambda$. However, since polynomials $f-\lambda$ for $\lambda\in\kk$ only differ in the constant part, it follows by \cite[Remark 3.5 and proof of Lemma 4.2]{HKV} that one can choose $N=N_\lambda$ independent of $\lambda$.
\end{proof}

\begin{rem}\label{r:uni}
Let $p_1,p_2\in\kk[t]$. Then $p_2\in\kk[p_1]$ if and only if for every $\lambda_1\in\kk$ there exists $\lambda_2\in\kk$ such that every zero of $p_1-\lambda_1$ is a zero of $p_2-\lambda_2$. Indeed, $p_1-\lambda_1$ has only simple zeros for infinitely many $\lambda_1$, in which case $\{p_1-\lambda_1=0\}\subseteq \{p_2-\lambda_2=0\}$ implies that $p_1-\lambda_1$ divides $p_2-\lambda_2$. Then the claim follows from the division algorithm in $\kk[t]$ by induction on $\deg p_2$.
\end{rem}

\begin{thm}\label{t:eig}
For $f,g\in\ax$ the following are equivalent:
\begin{enumerate}[(i)]
\item each \eig of $f$ is contained in an \eig of $g$;
\item there exist $p\in\kk[t]$ and nonzero $a,h\in\ax$ such that $g=p(h)$ and $fa=ah$.
\end{enumerate}
\end{thm}

\begin{proof}
(ii)$\Rightarrow$(i) By Lemma \ref{l:det},
$$h(\gX^{n})=a(\gX^{n})^{-1}f(\gX^{n})a(\gX^{n})$$
for all large enough $n$, and thus
$$\det(h(\gX^{n})-\lambda I)=\det(f(\gX^{n})-\lambda I)$$
for all $\lambda\in\kk$ and $n\in\N$. Hence $L_\lambda(f)=L_\lambda(h)$ for all $\lambda\in\kk$. Since every univariate polynomial over $\kk$ factors into linear factors, each \eig of $f$ is contained in an \eig of $p(h)$.

(i)$\Rightarrow$(ii) Assume that $f,g$ are nonconstant. Then $f=p_1(h_1)$ and $g=p_2(h_2)$ for some $p_1,p_2\in\kk[t]$ and non-composite $h_1,h_2\in\ax$ with $h_1(0)=0=h_2(0)$. By Corollary \ref{c:irr} there is a cofinite set $\Lambda\subseteq\kk$ such that $L_\lambda(h_1)\cap \mat{n}^d$ and $L_\lambda(h_2)\cap \mat{n}^d$ are reduced and irreducible hypersurfaces for all $\lambda\in\Lambda$ and large enough $n\in\N$. Since \eigs of $f$ are contained in \eigs of $g$, there are infinitely many pairs $(\lambda_1,\lambda_2)\in\Lambda^2$ such that $L_{\lambda_1}(h_1)=L_{\lambda_2}(h_2)$. By comparing
$$\det(h_1(\gX^{(n)})-\lambda_1 I),\quad \det(h_2(\gX^{(n)})-\lambda_2 I)$$
one can replace $h_2$ with $\alpha h_2+\beta$ for some $\alpha\in\kk\setminus\{0\}$ and $\beta\in\kk$ (and change $p_2$ accordingly) so that
\begin{equation}\label{e:det}
\det(h_1(\gX^{(n)})-\lambda I)=\det(h_2(\gX^{(n)})-\lambda I)
\end{equation}
for all $\lambda\in\kk$ and $n\in\N$. By \cite[Theorem 4.3]{HKV}, $h_1-\lambda$ and $h_2-\lambda$ are stably associated for all $\lambda\in \Lambda$. Let $\de=\deg h_1$. By Lemma \ref{l:cohn} there exist nonzero $a_\lambda,b_\lambda\in \ax$ of degree less than $\de$ for $\lambda\in \Lambda$ such that
\begin{equation}\label{e:sa}
(h_1-\lambda)a_\lambda=b_\lambda (h_2-\lambda).
\end{equation}
Since \eqref{e:sa} is a linear system in $(a_\lambda,b_\lambda)$ with a rational parameter $\lambda$, there exist nonzero $A,B\in \kk[t]\otimes \ax$ of degree (with respect to $x$) less than $\de$ such that
$$(h_1-t)A=B(h_2-t).$$
By looking at the degree of $A$ with respect to $t$ one can find $C\in \kk[t]\otimes \ax$ such that $a:=A-C(t-h_2)\in\ax$. Note that $a\neq0$ since $\deg A<\de=\deg h_2$. Letting $b:=B-(t-h_1)C$ we obtain
\begin{equation}\label{e:sa1}
(h_1-t)a=b(h_2-t).
\end{equation}
By comparing degrees with respect to $t$ in \eqref{e:sa1} we get $b\in\ax$ and consequently $a=b$. For $h:=p_1(h_2)$ we thus have
$$fa=p_1(h_1)a=ap_1(h_2)=ah.$$
Finally, since for every $\lambda_1\in\kk$ there exists $\lambda_2\in\kk$ such that
$$L_{\lambda_1}(p_1(h_2))=L_{\lambda_1}(h)=L_{\lambda_1}(f)\subseteq L_{\lambda_2}(g)=L_{\lambda_2}(p_2(h_2))$$
and $\det h_2(\gX^{(n)})$ is nonconstant for large $n$ by Lemma \ref{l:det}, Remark \ref{r:uni} implies $p_2=p\circ p_1$ for some $p\in\kk[t]$.
\end{proof}

\begin{cor}
Let $f,g\in\ax$. Then \eigs of $f$ and $g$ coincide if and only if there is a nonzero $a\in\ax$ such that $fa=ag$.
\end{cor}

\begin{proof}
If \eigs of $f$ and $g$ coincide, then $f=p_1(h_1)$, $g=p_2(h_2)$ and $h_1a=ah_2$ for $0\neq a,h_1,h_2\in\ax$ as in the proof of Theorem \ref{t:eig}. Furthermore,
$$L_\lambda (p_1(h_1))=L_\lambda (p_2(h_2))=L_\lambda (p_2(h_1))$$
implies $p_1=p_2$ and therefore $fa=ag$. For the converse see the proof of (ii)$\Rightarrow$(i) in Theorem \ref{t:eig}.
\end{proof}

\begin{exa}
Let
$$f=x_1+x_2+x_1x_2^2,\quad g=x_1+x_2+x_2^2x_1,\quad a=1+x_1^2+x_1x_2+x_2x_1+x_1x_2^2x_1.$$
Then $fa=ag$, so \eigs of $f$ and $g$ coincide. Note that $\deg a>\deg f$. While
$$f(1+x_2x_1)=(1+x_1x_2)g$$
holds, which complies with Lemma \ref{l:cohn}, there is no $b\in\ax$ such that $fb=bg$ and $\deg b\le\deg f$.
\end{exa}

\section{\Wqc polynomials}\label{s5}

On the free $\R$-algebra $\px$ there is a unique involution $*$ satisfying $x_j^*=x_j$. A noncommutative polynomial $f\in\px$ is {\bf symmetric} if $f^*=f$. Let $\all=\bigcup_{n\in\N}\sm{n}^d$. Then $f$ is symmetric if and only if $f(X)\in\mathbb{S}^1$ for all $X\in\all$. By $A\succ0$ (resp. $A\succeq0$) we denote that $A\in\mathbb{S}^1$ is positive definite (resp. semidefinite).

Let $f\in\px$ be symmetric. As in \cite{HM} (cf. \cite{HKMV}) we define its {\bf positivity domain},
$$\cD(f)=\bigcup_{n\in\N} \cD_n(f)$$
where $\cD_n(f)$ is the closure of the connected component of
$$\{X\in\sm{n}^d\colon f(X) \succ0 \}$$
containing the origin $0^d\in\sm{n}^d$. It is known \cite{HM} that $\cD(f)$ is convex (i.e., $\cD_n(f)$ is convex for all $n\in\N$) if and only if $\cD(f)$ is the solution set of a linear matrix inequality. We will require the following version of \cite[Theorem 1.5]{HKMV}.

\begin{prop}\label{p:hkmv}
Let $f\in\px$ be symmetric and irreducible over $\C$, with $f(0)=0$. If $\cD(1-f)$ is proper and convex, then
\begin{equation}\label{e:conc}
f=\ell_0+\ell_1^2+\cdots+\ell_m^2
\end{equation}
for some linear $\ell_0,\dots,\ell_m\in\px$.
\end{prop}

\def\tell{\tilde{\ell}}
\def\tf{\tilde{f}}

\begin{proof}
Let $\uly=(y_1,\dots,y_d)$ and $\uly^*=(y_1^*,\dots,y_d^*)$ be freely noncommuting variables, and consider $\py$ with the $\R$-linear involution $*$ sending $y_j$ to $y_j^*$ and acting on $\C$ as the complex conjugate. Since $f\in\px$ is symmetric and irreducible over $\C$, the noncommutative polynomial $\tf:=f(y_1+y_1^*,\dots,y_d+y_d^*)\in\py$ is hermitian and irreducible in $\py$. The positivity domain of $1-\tf$ (see \cite{HKMV}) is the union over $n\in\N$ of closures of connected components of
$$\{(Y,Y^*)\in\mc{n}^d\times\mc{n}^d\colon I-f(Y_1+Y_1^*,\dots,Y_d+Y_d^*)\succ0 \}$$
containing the origin. Furthermore, as $\cD(1-f)$ is proper and convex, the standard embedding of hermitian $n\times n$ matrices into symmetric $(2n)\times (2n)$ matrices implies that $\cD(1-\tf)$ is also proper and convex. Therefore
$$\tf=\tell_0+\sum_{k>0}\tell_k^*\tell_k$$
for some linear $\tell_k\in\py$ by \cite[Theorem 1.5]{HKMV}. Note that $f=\tf(x/2,x/2)$. Since $\tf$ is hermitian, $\tell$ is hermitian, so $\tell_0(x/2,x/2)$ is symmetric. Furthermore,
$$\tell_k^*\tell_k
=(\re \tell_k)^2+(\im \tell_k)^2
+i[\re \tell_k,\im \tell_k]$$
for $k>0$; since $f$ is symmetric, $\sum_{k>0}\tell_k(x/2,x/2)^*\tell_k(x/2,x/2)$ is a sum of squares in $\px$. Hence $f$ is of the form \eqref{e:conc}.
\end{proof}

\begin{rem}\label{r:lmi}
If $f$ is of the form \eqref{e:conc}, then it is easy to present $\cD(1-f)$ as the solution set of a linear matrix inequality, so $\cD(1-f)$ is convex.
\end{rem}

\begin{lem}\label{l:quad}
Let $h=\ell_0+\sum_{k>0}\ell_k^2$ for some linear $\ell_k\in\px$, and let $\ult=(t_1,\dots,t_d)$ be the coordinates of $\R^d$.
\begin{enumerate}[(i)]
	\item If $\beta>0$, then $h+\beta$ is a sum of squares in $\px$ if and only if $h(\ult)+\beta$ is a sum of squares in $\R[\ult]$.
	\item If $\cD_1(\alpha-h)\subseteq\cD_1(\beta+h)$ for some $\alpha,\beta>0$, then $\beta+h$ is a sum of squares. 
\end{enumerate}
\end{lem}

\begin{proof}
(i) Observe that $h+\beta$ has a unique representation $h+\beta=v^* S v$, where $S\in\sm{d+1}$ and $v^*=(1,x_1,\dots,x_d)$. It is easy to see that $h(\ult)+\beta$ is a sum of squares in $\R[t]$ if and only if $S\succeq0$, which is further equivalent to $h+\beta$ being a sum of squares in $\px$.

(ii) Since $\cD_1(\alpha-h)$ is convex, we have
$$h(\tau)\le\alpha\ \Rightarrow\ h(\tau)\ge -\beta$$
for all $\tau\in\R^d$. That is, an upper bound on $h(\ult)$ implies a lower bound on $h(\ult)$, which is clearly possible only if $h(\tau)\ge -\beta$ for all $\tau\in\R^d$. Since $h(\ult)+\beta$ is a quadratic nonnegative polynomial, it is a sum of squares in $\R[\ult]$. Now (ii) follows by (i).
\end{proof}

Recall that a symmetric $f\in\px$ with $f(0)=0$ is {\bf \wqc} if there exists $\ve>0$ such that $\cD(\lambda-f)$ is convex for every $\lambda\in (0,\ve)$.

\begin{thm}\label{t:wqc}
Le $f\in\px$ be symmetric with $f(0)=0$. The following are equivalent:
\begin{enumerate}[(i)]
\item $f$ is \wqc;
\item $\cD(\lambda-f)$ is convex for every $\lambda> 0$;
\item $-f$ is a sum of hermitian squares; or
\begin{equation}\label{e:cc}
f=p(\ell_0+\ell_1^2+\cdots+\ell_m^2)
\end{equation}
for $p\in\R[t]$ with $p(0)=0$ and linear $\ell_0,\dots,\ell_m\in\px$ satisfying one of the following:
\begin{enumerate}[(a)]
	\item $p(\tau)\le 0$ for $\inf_{\R^d} (\ell_0+\ell_1^2+\cdots+\ell_m^2)<\tau<0$,
	\item $\ell_k=0$ for all $k>0$.
\end{enumerate}
\end{enumerate}
\end{thm}

\begin{proof}
(ii)$\Rightarrow$(i) Clear.

(i)$\Rightarrow$(iii) Let $\ve>0$ be such that $\cD(\lambda-f)$ is convex for every $\lambda\in (0,\ve)$. If $\cD(\lambda-f)=\all$ for all such $\lambda$, then $-f(X)$ is positive semidefinite for every $X\in\all$, so $-f$ is a sum of hermitian squares by \cite{Hel,McC}. Otherwise we can without loss of generality assume that $\cD(\lambda-f)\neq\all$ for $\lambda\in (0,\ve)$. If $\lambda-f$ is irreducible over $\C$ for some such $\lambda$, then $f$ is of the form \eqref{e:conc} by Proposition \ref{p:hkmv}, and (a) holds with $p=t$. If $\lambda-f$ factors in $\px$ for all $\lambda\in(0,\ve)$, then $f=p(h)$ for some $p\in\R[t]$ and a non-composite $h\in\px$ with $p(0)=0=h(0)$ by Theorem \ref{t:bertini}. Since $f$ is symmetric, $h$ is also symmetric because it is unique up to a scalar multiple. Furthermore, $-p$ is not a sum of squares since $-f$ is not a sum of hermitian squares.

Let us introduce some auxiliary notation. If $\lambda-p$ attains a negative value on $(0,\infty)$, let $\pi_\lambda\ge0$ be such that
$$(\lambda-p)|_{[0,\pi_\lambda]}\ge0,\qquad
\exists\ve'>0\colon (\lambda-p)|_{(\pi_\lambda,\pi_\lambda+\ve')}< 0.$$
If $\lambda-p$ attains a negative value on $(-\infty,0)$, let $\nu_\lambda\le0$ be such that
$$(\lambda-p)|_{[\nu_\lambda,0]}\ge0,\qquad
\exists\ve'>0\colon (\lambda-p)|_{(\nu_\lambda-\ve',\nu_\lambda)}< 0.$$
Then $\pi_\lambda,\nu_\lambda$ are zeros of $\lambda-p$ and strictly monotone functions in $\lambda$, continuous for $\lambda$ close to $0$.

We distinguish two cases. First suppose that $-p$ is nonnegative on $(-\infty,0)$ or $(0,\infty)$. By replacing $p(t),h$ with $p(-t),-h$ if necessary, we can assume that $-p$ attains a negative value on $(0,\infty)$. Then
$$\cD(\lambda-f)=\cD(\pi_\lambda-h)$$
for all small enough $\lambda>0$. Since $h$ is not composite, $\pi_\lambda-h$ is irreducible for all but finitely many $\lambda$ by Theorem \ref{t:bertini}. Because $\cD(\lambda-f)$ is convex, $h$ is of the form \eqref{e:conc} by Proposition \ref{p:hkmv}, so (a) holds.

Now suppose that $-p$ attains negative values on$(-\infty,0)$ and $(0,\infty)$. Then
\begin{equation}\label{e:int}
\cD(\lambda-f)=\cD(-\nu_\lambda+h)\cap\cD(\pi_\lambda-h)
\end{equation}
for all small enough $\lambda>0$. Suppose that one the sets $\cD(-\nu_\lambda+h)$ and $\cD(\pi_\lambda-h)$ is contained in the other. By replacing $p(t),h$ with $p(-t),-h$ if necessary, we can assume that $\cD(\pi_\lambda-h)\subseteq \cD(-\nu_\lambda+h)$. Since $h$ is not composite, $\pi_\lambda-h$ is irreducible for all but finitely many $\lambda$, so $h$ is of the form \eqref{e:conc} by convexity of $\cD(\lambda-f)$ and Proposition \ref{p:hkmv}. By Lemma \ref{l:quad}, $-\nu_\lambda+h$ is a sum of squares. If $\mu=-\lim_{\lambda\downarrow 0}\nu_\lambda$, then $\mu+h$ is nonnegative on $\R^d$ and $-p$ is nonnegative on $[-\mu,0]$, so (a) holds. 
Finally we are left with the scenario where the intersection \eqref{e:int} is irredundant. Then 
$\cD(-\nu_\lambda+h)$ and $\cD(\pi_\lambda-h)$ are both convex by \cite[Corollary 1.2]{HKMV}. By Proposition \ref{p:hkmv} we conclude that $h$ is linear, so (b) holds.

(iii)$\Rightarrow$(ii) If $-f$ is a sum of hermitian squares, then $\cD(\lambda-f)=\all$ for every $\lambda>0$. Otherwise let $f$ be as in \eqref{e:cc}. Then $\cD(\lambda-f)$ equals one of
$$\all,\quad\cD(\pi_\lambda-h),\quad
\cD(-\nu_\lambda+h),\quad
\cD(-\nu_\lambda+h)\cap\cD(\pi_\lambda-h),$$
depending on the existence of $\nu_\lambda,\pi_\lambda$. Note that $\cD(\pi_\lambda-h)$ is always convex by Remark \ref{r:lmi}. If (b) holds, then $\cD(\lambda-f)$ is convex for $\lambda>0$ since $h$ is linear and intersection of convex sets is again convex. If (a) holds, then $\nu_\lambda\le \inf_{\R^d} (\ell_0+\sum_{k>0}\ell_k^2)$, so $\cD(-\nu_\lambda+h)=\all$ and $\cD(\lambda-f)$ is convex for $\lambda>0$.
\end{proof}

\begin{rem}
Few comments on the condition (a) in Theorem \ref{t:wqc} are in order. Let
$$\mu=\inf_{\R^d} (\ell_0+\ell_1^2+\cdots+\ell_m^2).$$
Then $\mu>-\infty$ if and only if $\ell_0$ lies in the linear span of $\ell_1,\dots,\ell_m$; more precisely, if $\ell_1,\dots,\ell_m$ are linearly independent and $\ell_0=\alpha_1\ell_1+\cdots+\alpha_m\ell_m$, then $-4\mu=\alpha_1^2+\cdots+\alpha_m^2$. This follows from considering $\ell_0+\ell_1^2+\cdots+\ell_m^2+\mu=v^*Sv$ for $v^*=(1,x_1,\dots,x_d)$ and $S\succeq0$ as in the proof of Lemma \ref{l:quad}(i). Furthermore, using an algebraic certificate for nonnegativity \cite[Prop 2.7.3]{Mar}, the condition (a) can also be stated as follows. Let $S\subset\R[t]$ be the convex cone of sums of (two) squares. If $\mu=-\infty$, then
$$\sup_{(-\infty,0]}p=p(0)=0 \quad\iff\quad p\in t(S-tS);$$
and if $-\infty<\mu\le0$, then
$$\max_{[\mu,0]}p=p(0)=0 \quad\iff\quad p\in t(S-tS+(t-\mu)(S-tS)).$$
\end{rem}

\begin{rem}
Another aspect of Theorem \ref{t:wqc} is the following. Proposition \ref{p:hkmv} states that every irreducible symmetric polynomial with a convex positivity domain is quadratic (and concave). On the other hand, there is no shortage of reducible symmetric polynomials that contain a factor of degree at least $3$ and have convex positivity domain; see \cite[Sections 5 and 6]{HKMV}. However, if the constant term of such a polynomial is slightly perturbed, then its positivity domain is no longer convex by Theorem \ref{t:wqc}.
\end{rem}


\end{document}